\documentclass[12pt]{amsart}

\makeatletter
\@namedef{subjclassname@2020}{\textup{2020} Mathematics Subject Classification}
\makeatother

\usepackage[numbers,sort&compress]{natbib}

\usepackage{amsmath}

\newtheorem{theorem}{Theorem}[section]
\newtheorem{lemma}[theorem]{Lemma}

 \theoremstyle{definition}
\newtheorem{definition}[theorem]{Definition}

\theoremstyle{remark}
\newtheorem{remark}[theorem]{Remark}

\numberwithin{equation}{section}

\begin{document}

\title[The Oblique boundary value problem]
{The oblique boundary value problem for degenerate Hessian quotient type equations}

\author{Ni Xiang}
\address{Faculty of Mathematics and Statistics, Hubei Key Laboratory of Applied Mathematics, Hubei University,  Wuhan 430062, P.R. China}
\email{nixiang@hubu.edu.cn}

\author{Yuni Xiong$^\ast$}
\address{Faculty of Mathematics and Statistics, Hubei Key Laboratory of Applied Mathematics, Hubei University,  Wuhan 430062, P.R. China}
\email{202121104011767@stu.hubu.edu.cn}

\author{Lina Zheng}
\address{Faculty of Mathematics and Statistics, Hubei Key Laboratory of Applied Mathematics, Hubei University,  Wuhan 430062, P.R. China}
\email{202221104011317@stu.hubu.edu.cn}


\thanks{This research was supported by funds from the National Natural Science Foundation of China No. 11971157, No. 12426532, No. 12571213}
\thanks{$\ast$ Corresponding author}

\begin{abstract}
In this paper, we investigate the oblique boundary value problem for degenerate Hessian quotient type equations in a smooth bounded domain. Without imposing any geometric restrictions on the domain, we establish the a priori estimates and derive the existence and uniqueness of admissible $C^{1,1}$ solutions under the condition $f^{\frac{1}{k-l}}\in C^{1,1}(\overline{\Omega}\times\mathbb{R})$.
\end{abstract}

\maketitle
\section{Introduction}
This article is devoted to the study of oblique boundary value problems for degenerate Hessian quotient type equations
\begin{equation}\label{eq07081}
\left\{
\begin{aligned}
  &\frac{\sigma_{k}}{\sigma_{l}}(\gamma\triangle u I-D^{2}u)=f(x,u),\,\,\,\text{in}\,\;\Omega,\\
  &G(x,u,Du)=0,\,\,\,\,\quad\quad\quad\quad\quad\text{on}\,\;\partial\Omega,
\end{aligned}
\right.
\end{equation}
where $\Omega$ is a smooth bounded domain in $\mathbb{R}^{n}$, $\gamma\geq1$, $0\leq l<k\leq n$, $Du$ and $D^2 u$ denote the gradient vector and the Hessian matrix of the solution $u$. Here $f(x,u)$ is a non-negative function defined on $\overline{\Omega}\times\mathbb{R}$.

For convenience, let $x,z,p$ be points in $\Omega, \mathbb{R}, \mathbb{R}^n$, respectively. The boundary value condition $G\in C^1(\partial\Omega \times \mathbb{R}\times \mathbb{R}^n)$ is said to be strictly oblique with respect to $u$, if
\begin{equation}\label{eq1002}
G_p(x,u,Du)\cdot \nu \geq \beta_0>0,\,\,\text{for all}\,\;(x,u,Du)\in \partial\Omega \times \mathbb{R}\times \mathbb{R}^n,
\end{equation}
where $\nu$ is the unit inner normal vector field on $\partial\Omega$ and $\beta_{0}$ is a positive constant.
If $G_{p}\cdot\nu > 0$ on all of $\partial\Omega\times\mathbb{R}\times\mathbb{R}^{n}$, we will simply refer to $G$ as oblique.
As in Jiang-Trudinger-Xiang \cite{JTX16}, we assume $G\in C^{2}(\partial\Omega\times\mathbb{R}\times\mathbb{R}^{n})$ is concave in $p$ if $G_{ pp} \leq0$ on $\partial\Omega\times\mathbb{R}\times\mathbb{R}^{n}$. This includes the quasilinear oblique case, when $G_{ pp}= 0$
\begin{equation}\label{eq1003}
G= \beta(x,u)\cdot p-\varphi(x,u)=0,
\end{equation}
where $p=Du$, $\beta(x,u)=G_p$ and $\varphi(x,u)$ are defined on $\partial\Omega\times\mathbb{R}$.
In this paper, we focus on the semilinear oblique  boundary condition where $\beta$ in \eqref{eq1003} is independent of $u$ s.t.
\begin{equation}\label{eq201}
  G=\beta(x)\cdot p-\varphi(x,u)=0,
\end{equation}
which is equivalent to $u_{\beta}=\varphi(x,u)$. The standard example of \eqref{eq201} is the semilinear Neumann boundary condition, where $\beta=\nu$ on $\partial\Omega$.

Next, we recall some definitions about Hessian equations.
\begin{definition}\label{df101}
Assume $\lambda=(\lambda_1,\dots,\lambda_n)\in\mathbb{R}^n$, the elementary symmetric function for $k=1,2,\cdots,n$ is defined by
\begin{equation*}
\sigma_k(\lambda)= \sum _{1 \le i_1 < i_2 <\cdots<i_k\leq
n}\lambda_{i_1}\lambda_{i_2}\cdots\lambda_{i_k},
\end{equation*}
 specially, denote $\sigma_0=1$ and $\sigma_k=0$ for $k>n$ or $k<0$.
\end{definition}
\begin{definition}\label{df102}
We say $u\in C^2(\Omega)$ is $k$-admissible if $\lambda(\gamma\triangle uI-D^2 u)(x)\in \Gamma_k$ for any $x\in \Omega$, and the G$\mathring{\text{a}}$rding's cone $\Gamma_k$ is
\begin{equation*}
\Gamma_k=\{\lambda\in \mathbb{R}^n:\sigma_i(\lambda)>0, \forall \, 1\leq i \leq k\}.
\end{equation*}
\end{definition}

Hessian quotient type equations
\begin{equation}\label{eq03091}
\frac{\sigma_{k}}{\sigma_{l}}(\gamma\triangle u I-D^{2}u)=f
\end{equation}
originate from many geometric problems and play an important role in conformal geometry, complex geometry. Equation \eqref{eq03091} is called the $(n-1)$ Monge-Amp\`ere equation which was introduced by Harvey-Lawson in a series of papers \cite{HL11, HL12, HL13} when $k=n$, $l=0$ and $\gamma=1$. And more information can be consulted in \cite{Guan21,FW,FWW,Ger03,P,TW17,To19} and references therein.

For non-degenerate cases, 
the Dirichlet problem for basic Hessian equations has been well studied  by Caffarelli-Nirenberg-Spruck \cite{CNS85}, Ivochkina \cite{Iv87} and Trudinger \cite{Tr95}.
Some works about the Neumann problem for basic Hessian equations can be found in \cite{CZ21,MQ19,QX19,Tr87}. Dong-Wei \cite{DW22} have established the a priori estimates to equation \eqref{eq03091} 
when  $\gamma=1$, $\lambda(\triangle uI-D^2u)(x) \in\Gamma_{k+1}$. Then
Chen-Dong-Han \cite{CDH23} showed that the existence and uniqueness of the $k$-admissible solution to equation \eqref{eq03091} with $\gamma=1$, $0\leq l<k<n$.

We turn our attention to the degenerate case, where such estimates become more delicate and require refined techniques. The a priori $C^{1,1}$ estimates remain pivotal in establishing the regularity of solutions for these degenerate problems.
For degenerate Monge-Amp\`ere equations, foundational work by Guan-Li \cite{GL94,GL97} on the degenerate Weyl problem and Gauss curvature measure established under the following conditions which are critical to global $C^{1,1}$ estimates
\begin{equation}\label{eq06302}
\Delta(f^{1/(n-1)})\geq-A, \quad |D(f^{1/(n-1)})|\leq A,
\end{equation}
where $A$ is a positive constant.
Guan \cite{Guan97} found that conditions \eqref{eq06302} are also sufficient for the Dirichlet problem of degenerate Monge-Amp\`ere equations with homogenous boundary. The non-homogeneous boundary case was resolved by Guan-Trudinger-Wang \cite{GTW99} under the stronger assumption $f^{\frac{1}{n-1}}\in C^{1,1}(\overline{\Omega})$, which implicitly enforces \eqref{eq06302}.
For the Dirichlet problem of degenerate Hessian equations,
Dong \cite{Dong06} established $C^{1,1}$ estimates to homogeneous boundary under conditions
\begin{equation*}
f^{\frac{1}{k-1}}\in C^{1,1}(\overline{\Omega}), \quad |D(f^{\frac{1}{k-1}})|\leq C f^{\frac{1}{2(k-1)}}.
\end{equation*}
Recently, $C^{1,1}$ estimates to convex solutions with non-homogenous boundary in the strictly convex bounded domain was solved by Jiao-Wang \cite{JW24}, requiring the condition $f^{\frac{1}{k-1}}\in C^{1,1}(\overline{\Omega})$.
For more references, the reader
may see \cite{CNS86,GuanB,WangXu,Mei22} and related works cited there.

Therefore, it is of natural interest to consider existence results for the oblique boundary value problem. The oblique boundary value problem for  Hessian type equations arises naturally in the theory of fully nonlinear elliptic equations, through its applications in
conformal geometry, optimal transportation, and geometric optics.
Wang \cite{Wang92} established the existence of solutions for two dimensional Monge-Amp\`ere equations and subsequently considered general dimensions. Urbas \cite{Urbas95,Urbas96,Urbas98} studied Monge-Amp\`{e}re equations, two dimensional Hessian equations and curvature equations of the oblique boundary value problem.
The oblique boundary value problem of linear and quasilinear elliptic equations can be seen in the book \cite{Lie}.
Lieberman-Trudinger \cite{LT86} investigated the nonlinear oblique boundary value problem for quasilinear and fully nonlinear uniformly elliptic partial differential equations under some structure conditions. For augmented Hessian equations with the oblique boundary value problem
\begin{eqnarray*}
\begin{cases}
  F(D^{2}u-A(x,z,p))=B(x,z,p),\,\,\,\text{in}\,\;\Omega,\\
  G(x,z,p)=0,\,\,\,\,\quad\quad\quad\quad\quad\quad\quad\quad\text{on}\,\;\partial\Omega,
\end{cases}
\end{eqnarray*}
a series works of Jiang-Trudinger \cite{JT1,JT2,JT3} considered existence results by formulating some conditions on $A$ and $F$,  which provided a comprehensive framework for studying the oblique boundary value problem for a class of fully nonlinear equations.
Then Wang \cite{Wph22} derived global gradient estimates for admissible solutions to Hessian equations in $\mathbb{R}^n$ with the oblique boundary value problem.

Following these developments, we investigate the oblique boundary value problem for degenerate Hessian quotient type equations
under the assumption of subsolutions $\underline{u}$ with $f^{\frac{1}{k-l}}\in C^{1,1}(\overline{\Omega}\times\mathbb{R})$, $0\leq l<k<n$ and $\gamma\geq 1$.

\begin{theorem}\label{Th1}
Let $\Omega\subset\mathbb{R}^{n}$ be a bounded $C^{4}$ domain, $\varphi(x,u)$ be a smooth function, $f^{\frac{1}{k-l}}\in C^{1,1}(\overline{\Omega}\times\mathbb{R})$, $f_u\geq 0$, $\varphi_u\geq \gamma_0>0$. If there exists a subsolution $\underline{u}\in  C^2(\overline{\Omega})$ satisfying
\begin{equation}\label{eq091801}
\left\{
\begin{aligned}
  &\frac{\sigma_{k}}{\sigma_{l}}(\gamma\triangle \underline{u} I-D^{2}\underline{u})\geq f(x,\underline{u}),\quad \text{in}\,\;\Omega,\\
  &\underline{u}_\beta=\varphi(x,\underline{u}),\quad\quad \quad\quad\quad\quad\quad\,\, \text{on}\,\;\partial\Omega,
\end{aligned}
\right.
\end{equation}
and a unique admissible solution $u\in  C^4(\overline{\Omega})$ satisfying
\begin{equation}\label{eq1001}
\left\{
\begin{aligned}
  &\frac{\sigma_{k}}{\sigma_{l}}(\gamma\triangle u I-D^{2}u)=f(x,u),\quad \text{in}\,\;\Omega,\\
  &u_\beta=\varphi(x,u),\quad\quad \quad\quad\quad\quad\quad\,\,\text{on}\,\;\partial\Omega,
\end{aligned}
\right.
\end{equation}
where $\langle\beta,\nu\rangle\geq\beta_0>0$, $\gamma\geq 1$, $0\leq l<k<n$. Then
\begin{equation}\label{eq091901}
|u|_{C^{1,1}(\overline{\Omega})}\leq C,
\end{equation}
where the positive constant $C$ depends on $n,\ k,\ l,\ \gamma_0,\ \gamma,\ \beta_0,\ \overline{\Omega}$, $|\varphi|_{C^3}$, $|\underline{u}|_{C^0}$, $|f^{\frac{1}{k-l}}|_{C^{1,1}}$.
\end{theorem}

\begin{remark}
The admissible subsolution hypothesis and conditions $f_u\geq 0$, $\varphi_u\geq \gamma_0>0$ are required only for the maximum modulus estimates. Furtherly, when $\gamma>1$ in problem \eqref{eq1001}, the existence theorem extends to the case $k=n$.
\end{remark}

\begin{remark}
When the matrix $\left\{\gamma \triangle u I-D^2 u\right\}$ ($\gamma \geq 1$) is replaced by $\left\{\gamma \triangle u I+D^2 u\right\}$ ($\gamma > 0$) in problem \eqref{eq1001}, the existence theorem remains valid for the problem
\begin{equation}\label{eq091803}
\left\{
\begin{aligned}
  &\frac{\sigma_{k}}{\sigma_{l}}(\gamma\triangle u I+D^{2}u)=f(x,u),\,\,\,\text{in}\,\;\Omega,\\
  &u_\beta=\varphi(x,u),\quad\quad \quad\quad\quad\quad\quad\,\text{on}\,\;\partial\Omega,
\end{aligned}
\right.
\end{equation}
where $0\leq l<k\leq n$.
\end{remark}

The remainder of this paper is organized as follows. In Section 2, we introduce basic notations and establish preliminary results for Hessian quotient type equations. Section 3 is devoted to deriving $C^0$ estimates under the assumption of admissible subsolutions with $f_u \geq 0$ and $\varphi_u \geq \gamma_0 > 0$, followed by $C^1$ estimates under semilinear oblique boundary conditions. Finally, Section 4 presents the crucial second order estimates that complete the proof of Theorem \ref{Th1}.


\section{Preliminaries}
In this section, we recall some basic properties of Hessian quotient type equations, which can be found in \cite{C-P,CDH23,DW22,LT94,Sp05}.
For convenience, we define
$\lambda=(\lambda_{1},\lambda_{2},\cdots,\lambda_{n})$, $\eta=(\eta_{1},\eta_{2},\cdots,\eta_{n})$ are eigenvalue vectors of matrix $D^{2}u$ and $U=\gamma\triangle uI-D^{2}u$, respectively. Set
\begin{eqnarray*}
F(u_{ij})=\frac{\sigma_{k}(U_{ij})}{\sigma_{l}(U_{ij})},\quad F^{ij}=\frac{\partial F}{\partial u_{ij}},\quad F^{ij,rs}=\frac{\partial^{2}F}{\partial u_{ij}\partial u_{rs}},
\end{eqnarray*}
\begin{equation*}
  \widetilde{F}(u_{ij})=\left[\frac{\sigma_{k}(U_{ij})}{\sigma_{l}(U_{ij})}\right]^{\frac{1}{k-l}},\quad \widetilde{F}^{ij}=\frac{\partial \widetilde{F}}{\partial u_{ij}},\quad \widetilde{F}^{ij,rs}=\frac{\partial^{2}\widetilde{F}}{\partial u_{ij}\partial u_{rs}},
\end{equation*}
thus Hessian quotient type equations in problem (\ref{eq1001}) can be written as
\begin{equation}\label{eq0904201}
\widetilde{F}(u_{ij})=F^{\frac{1}{k-l}}(u_{ij})=\left[\frac{\sigma_{k}(U_{ij})}{\sigma_{l}(U_{ij})}\right]^{\frac{1}{k-l}}
=f^{\frac{1}{k-l}}(x,u)=\widetilde{f}(x,u).
\end{equation}
Differentiating equation \eqref{eq0904201} twice with respect to vector fields $\xi=(\xi^{1},\xi^{2},\cdots,\xi^{n})$, we have
\begin{eqnarray}
  F^{ij}u_{ij\xi}=\partial_{\xi}f=f_{x_{\xi}}+f_{z}u_{\xi},\label{eq0904203}\\
  \widetilde{F}^{ij}u_{ij\xi}=\partial_{\xi}\widetilde{f}=\widetilde{f}_{x_{\xi}}+\widetilde{f}_{z}u_{\xi},\label{eq0913202}
\end{eqnarray}
and
\begin{eqnarray}
  F^{ij,rs}u_{ij\xi}u_{rs\xi}+F^{ij}u_{ij\xi\xi}=\partial_{\xi\xi}f=f_{x_{\xi}x_{\xi}}+2f_{x_{\xi}z}u_{\xi}+f_{zz}u_{\xi}^{2}
  +f_{z}u_{\xi\xi},\label{eq0904204}\\
  \widetilde{F}^{ij,rs}u_{ij\xi}u_{rs\xi}+\widetilde{F}^{ij}u_{ij\xi\xi}=\partial_{\xi\xi}\widetilde{f}=\widetilde{f}_{x_{\xi}x_{\xi}}
   +2\widetilde{f}_{x_{\xi}z}u_{\xi}+\widetilde{f}_{zz}u_{\xi}^{2}+\widetilde{f}_{z}u_{\xi\xi}.\label{eq091902}
\end{eqnarray}




The following lemmas can be referenced to \cite{CDH23,LT94,Sp05}.

\begin{lemma}\label{le1}
If $\lambda=(\lambda_1,\dots,\lambda_n)\in\Gamma_k$ with $\lambda_1\geq\lambda_2\geq\cdots\geq\lambda_n$ for $1\leq k\leq n$,
then
\begin{equation*}
  \sigma_{k-1}(\lambda|k)\geq C(n,k)\sigma_{k-1}(\lambda),\quad\sigma_{k-1}(\lambda|k)\geq C(n,k)\sum_{i=1}^n\sigma_{k-1}(\lambda|i),
\end{equation*}
where $\sigma_{k-1}(\lambda|i)=\frac{\partial\sigma_k}{\partial \lambda_i}$ and $C(n,k)$ is a positive constant only depending on $n$ and $k$.
\end{lemma}
\begin{lemma}\label{le2}
For $\lambda\in\Gamma_k$, $n\geq k>l\geq 0$, $n\geq r>s\geq 0$, $k\geq r$ and $l\geq s$, we have
\begin{eqnarray*}
\left[\frac{\sigma_k(\lambda)/C_{n}^k}{\sigma_l(\lambda)/C_{n}^{l}}\right]^\frac{1}{k-l}\leq\left[\frac{\sigma_r(\lambda)/C_{n}^r}{\sigma_{s}(\lambda)/C_{n}^{s}}\right]^\frac{1}{r-s},
\end{eqnarray*}
and the equality holds if and only if $\lambda_1=\cdots=\lambda_n>0$.
\end{lemma}

\begin{lemma}\label{lem0913201}
If $\lambda=(\lambda_1,\lambda_2,\cdots,\lambda_n)\in\widetilde{\Gamma}_k$, then $\left[\frac{\sigma_{k}(\eta)}{\sigma_{l}(\eta)}\right]^{\frac{1}{k-l}}(0\leq l<k\leq n)$ is concave with respect to $\lambda$. Hence, for any $\xi\in\mathbb{R}^{n}$
\begin{equation}\label{eq07111}
  \sum_{i,j,r,s}\frac{\partial^{2}\left[\frac{\sigma_{k}(\eta)}{\sigma_{l}(\eta)}\right]}{\partial u_{ij}\partial u_{rs}}\xi_{ij}\xi_{rs}\leq\left(1-\frac{1}{k-l}\right)\frac{\left[\sum\limits_{i,j}\frac{\partial\left[\frac{\sigma_{k}(\eta)}{\sigma_{l}(\eta)}\right]}{\partial u_{ij}}\xi_{ij}\right]^{2}}{\frac{\sigma_{k}(\eta)}{\sigma_{l}(\eta)}},
\end{equation}
where $\widetilde{\Gamma}_{k}$ is a convex symmetric cone in $\mathbb{R}^{n}$, given by
$$\widetilde{\Gamma}_{k}=\left\{\lambda\in\mathbb{R}^{n}:\sigma_{m}(\eta)>0,\;\forall\; m=1,\cdots,k\right\}.$$
\end{lemma}

\begin{lemma}\label{le3}
If $\lambda=(\lambda_1,\lambda_2,\cdots,\lambda_n)\in\widetilde{\Gamma}_k$ with $\lambda_1\geq\lambda_2\geq\cdots\geq\lambda_n$, $\eta_i=\gamma \sum_{j=1}^n \lambda_j-\lambda_i$, $1\leq i\leq n$, for $0\leq l<k<n$ we obtain

\text{(1)}
$\frac{\partial\left[\frac{\sigma_k(\eta)}{\sigma_l(\eta)}\right]}{\partial\lambda_1}\leq\frac{\partial\left[\frac{\sigma_k(\eta)}{\sigma_l(\eta)}\right]}{\partial\lambda_2}
\leq\cdots\leq\frac{\partial\left[\frac{\sigma_k(\eta)}{\sigma_l(\eta)}\right]}{\partial\lambda_n},$

\text{(2)}
$\frac{\partial\left[\frac{\sigma_k(\eta)}{\sigma_l(\eta)}\right]}{\partial\eta_1}\geq\frac{\partial\left[\frac{\sigma_k(\eta)}{\sigma_l(\eta)}\right]}{\partial\eta_2}
\geq\cdots\geq\frac{\partial\left[\frac{\sigma_k(\eta)}{\sigma_l(\eta)}\right]}{\partial\eta_n},$

\text{(3)}
$\frac{\partial\left[\frac{\sigma_k(\eta)}{\sigma_l(\eta)}\right]}{\partial\lambda_i}\geq \frac{2\gamma-1}{n\gamma-1}C(n,k,l)\sum_{j=1}^{n}\frac{\partial\left[\frac{\sigma_k(\eta)}{\sigma_l(\eta)}\right]}{\partial\lambda_j},$

\text{(4)}
$\sum_{j=1}^{n}\frac{\partial\left[\frac{\sigma_k(\eta)}{\sigma_l(\eta)}\right]}{\partial\lambda_j}
\geq C(n,k,l)f^{1-\frac{1}{k-l}}.$
\end{lemma}

\begin{remark}
Compared to the classical Hessian equation with the internal matrix $D^2u$, the Hessian type equation involving $\gamma \triangle u I-D^2 u$ exhibits uniform ellipticity which is indispensable for subsequent the a priori estimates.
\end{remark}

\begin{remark}\label{0919rk1}
When the matrix $\left\{\gamma \triangle u I-D^2 u\right\}$ ($\gamma \geq 1$) is replaced by $\left\{\gamma \triangle u I+D^2 u\right\}$ ($\gamma > 0$), the analogues of Lemma \ref{lem0913201} and Lemma \ref{le3} still hold for $0\leq l<k\leq n$. Furthermore, the corresponding a priori estimates can be established through similar arguments, whose details we omit here.
\end{remark}

\section{$C^0$ and $C^1$ estimates}
In this section, we first prove $C^0$ estimates for admissible solutions to problem \eqref{eq1001} under the subsolution assumption with structural conditions $f_u\geq 0$, $\varphi_u\geq \gamma_0>0$,  then proceed to derive $C^1$ estimates.

\begin{theorem}\label{C0}
Assume $\varphi$ is a smooth function and $f_u\geq 0$, $\varphi_u\geq \gamma_0>0$.
If there exists a subsolution $\underline{u}\in C^{2}(\overline{\Omega})$ satisfying problem \eqref{eq091801} and an admissible solution $u\in C^{2}(\overline{\Omega})$ satisfying problem \eqref{eq1001} in a bounded $C^2$ domain $\Omega\subset\mathbb{R}^{n}$, then
\begin{equation*}
 \sup_{\overline{\Omega}} |u|\leq C,
\end{equation*}
where $C$ depends on $\gamma_0$, $|\underline{u}|_{C^0}$, $|\varphi|_{C^0}$.
\end{theorem}
\begin{proof}
From the $(k-l)$-degree homogeneity of the operator $F$, we derive $\sum_{i,j} F^{ij}u_{ij} \geq 0$. Consequently, the maximum of $u$ is attained at a boundary point $x_0 \in \partial\Omega$ with $u(x_0) > 0$, otherwise, the upper bound would be obtained. This implies
\begin{align*}
0\geq D_\beta u(x_0)
&=\varphi(x_0, u(x_0))-\varphi(x_0, 0)+\varphi(x_0, 0)\\
&=\varphi_u(x_0,tu(x_0))u(x_0)+\varphi(x_0,0)\\
&\geq\gamma_0u(x_0)+\varphi(x_0,0),
\end{align*}
where $0<t<1$. Then $u(x_0)\leq \frac{-\varphi(x_0,0)}{\gamma_0}\leq C$.

By \eqref{eq091801} and \eqref{eq1001}, we get the fact $F(u_{ij})-f(x,u)=0$ and $F(\underline{u}_{ij})-f(x,\underline{u})\geq0$,
which implies
\begin{align*}
0\leq &\,F(\underline{u}_{ij})-f(x,\underline{u})-F(u_{ij})+f(x,u)\\
=&\,\int_0^1F^{ij}(t\underline{u}_{ij}+(1-t)u_{ij})dt\cdot(\underline{u}_{ij}-u_{ij})-\int_0^1f_udt\cdot(\underline{u}-u).
\end{align*}
Under the condition $f_u\geq 0$, the maximum principle yields that $\underline{u}-u$ achieves its maximum at a boundary point $x_1\in \partial\Omega$ by applying the maximum principle.
Thus
\begin{align*}
0\geq D_\beta(\underline{u}-u)(x_1)&=\varphi(x_1,\underline{u}(x_1))-\varphi(x_1,u(x_1))\\
&=\varphi_u[x_1,tu(x_1)+(1-t)\underline{u}(x_1)](\underline{u}-u)(x_1).
\end{align*}
Since $\varphi_u>0$, for any $x\in\overline{\Omega}$
$$(\underline{u}-u)(x)\leq(\underline{u}-u)(x_1)\leq 0.$$
Hence we finish the proof.
\end{proof}

\begin{theorem}\label{C1}
Let $u\in C^{3}(\overline{\Omega})$ be an admissible solution of problem (\ref{eq1001}) in a bounded $C^3$ domain $\Omega\subset\mathbb{R}^{n}$. Assume that $0\leq f^{\frac{1}{k-l}}\in C^{1,1}(\overline{\Omega}\times\mathbb{R})$, the oblique boundary value condition $u_{\beta}=\varphi(x,u)$ satisfies $\left<\beta,\nu\right>\geq \beta_{0}>0$, where $\beta$ is a smooth unit vector field along $\partial\Omega$, $\nu$ is the unit inner normal vector field along $\partial\Omega$, $\varphi$ is a smooth function defined on $\overline{\Omega}\times\mathbb{R}$. Then for $\gamma\geq 1$, $0\leq l< k< n$
\begin{equation}\label{eq3002}
  \sup_{\overline{\Omega}}|Du|\leq C,
\end{equation}
where $C$ is a positive constant depending only on $n,k,l,\gamma,\overline{\Omega},\beta_{0},|\beta|_{C^{3}},|f^{\frac{1}{k-l}}|_{C^{1}},|\varphi|_{C^{3}}$.
\end{theorem}

\begin{proof}
Building upon interior gradient estimates established by Dong-Wei \cite{DW22}, we focus our analysis on deriving gradient estimates near the boundary. Denote $\Omega_{\mu}=\left\{x\in\Omega\;|\; d(x)=dist(x,\partial\Omega)\leq\mu\right\}$, here $\mu$ is a small positive constant to be determined later. Now extend $\beta$ to $\Omega_{\mu}$ smoothly such that $\left<\beta,Dd\right>=\cos\theta\geq \beta_{0}>0$.

Taking a unit normal moving frame along $\partial\Omega$ denoted by $\{e_{1},e_{2},\cdots,e_{n-1},\nu\}$, the oblique vector field decomposes as
\begin{equation}\label{eq3003}
  \beta=\beta_{\nu}\cdot\nu+\sum_{q=1}^{n-1}\beta_{q}e_{q},
\end{equation}
where $\beta_{\nu}=\left<\beta,\nu\right>=\cos\theta\geq \beta_{0}>0$, thus $\varphi(x,u)=u_{\nu}\beta_{\nu}+\sum_{q=1}^{n-1}u_{q}\beta_{q}$ on $\partial\Omega$. Defining $w(x)=u(x)-\frac{\varphi(x,u)d(x)}{\cos\theta}$, on $\partial\Omega$ we get
\begin{align}
  &w_{\nu}=-\sum_{q=1}^{n-1}w_{q}\frac{\beta_{q}}{\beta_{\nu}}, \label{eq3006} \\
  &w_{\nu}^{2}\leq\sin^{2}\theta|Dw|^{2}. \label{eq3007}
\end{align}
Set an auxiliary function
\begin{equation}\label{eq3008}
  \widetilde{G}(x)=\log\psi(x)+g(u)+\alpha_{1} d(x),
\end{equation}
where $\alpha_{1}$ is a sufficiently large positive constant and
\begin{equation}\label{eq3009}
  \psi(x)=|Dw|^{2}-\left(\sum_{i=1}^{n}w_{i}d_{i}\right)^{2}=\sum_{i,j=1}^{n}\left(\delta_{ij}-d_{i}d_{j}\right)w_{i}w_{j}
  =\sum_{i,j=1}^{n}C^{ij}w_{i}w_{j},
\end{equation}
\begin{equation}\label{eq3010}
  g(u)=\frac{1}{4}\ln\frac{1}{(3M-u)},\quad M=\max_{\overline{\Omega}} u,
\end{equation}
then we derive
$$\frac{1}{16M}\leq g' \leq \frac{1}{8M},\quad g''-2(g')^2 \geq \frac{1}{128M^2}.$$
Suppose $\widetilde{G}(x)$ achieves its maximum value at $x_{0}\in\overline{\Omega}_{\mu}$. In view of interior gradient estimates, the analysis reduces to the following two cases.

$\mathbf{Case\;\mathrm{\mathbf{I}}}$: $x_{0}\in\partial\Omega$.
We choose a coordinate around $x_{0}$ such that $\nu=\frac{\partial}{\partial x_{n}}$ and $\frac{\partial}{\partial x_{s}}$ are tangent to $\partial\Omega$, then
\begin{equation*}
  \frac{\partial d}{\partial x_{n}}=1,\,\frac{\partial d}{\partial x_{s}}=0,\,\frac{\partial^{2} d}{\partial x_{n}\partial x_{i}}=0\,(1\leq i\leq n), \,\frac{\partial^{2} d}{\partial x_{s}\partial x_{q}}=-\kappa_{s}\delta_{sq}\,(1\leq s,q\leq n-1),
\end{equation*}
here $\kappa_{s}$ are principal curvatures of $\partial\Omega$ at $x_{0}$.
By virtue of \eqref{eq3007}
\begin{equation*}
  \psi=|Dw|^{2}-(w_nd_{n})^{2}\geq \cos^2\theta |Dw|^{2}\geq \beta_0^2|Dw|^{2}>0.
\end{equation*}
A direct calculation yields
\begin{equation}\label{305}
  w_{\nu s}=w_{s\nu}+\sum_{i=1}^{n}\mathcal{K}_{is}w_{i},
\end{equation}
where $\mathcal{K}_{is}$ denoted by the Weingarten matrix of the boundary with respect to $\nu$.
We first deal with $\psi_{\nu}$ for later computations
\begin{equation}\label{307}
  \psi_{\nu}(x_{0})=\psi_{n}
  =-2\sum_{s,q=1}^{n-1}w_{s}w_{sq}\frac{\beta_{q}}{\beta_{\nu}}-2\sum_{s,q=1}^{n-1}w_{s}w_{q}\left(\frac{\beta_{q}}{\beta_{\nu}}\right)_{s}
  -2\sum_{s=1}^{n-1}\sum_{i=1}^n\mathcal{K}_{si}w_{s}w_{i}.
\end{equation}
At $x_{0}$
\begin{align}
0&=\widetilde{G}_{s}=\frac{\psi_{s}}{\psi}+g'u_{s}+\alpha_{1}d_{s}=\frac{\psi_{s}}{\psi}+g'u_{s},\label{308} \\
0&\geq\widetilde{G}_{n}=\frac{\psi_{n}}{\psi}+g'u_{n}+\alpha_{1}. \label{309}
\end{align}
By \eqref{307}-\eqref{309}
\begin{equation}\label{312}
 \begin{aligned}
  0\geq&\;-\frac{2w_{n}\sum_{s=1}^{n-1}\sum_{i=1}^nw_{i}\beta_{s}\kappa_i\delta_{is}}{\beta_{n}\psi}-
  \frac{2\sum_{s,q=1}^{n-1}w_{s}w_{q}\left(\frac{\beta_{q}}{\beta_{\nu}}\right)_{s}}{\psi}\\
  &\;-\frac{2\sum_{s=1}^{n-1}\sum_{i=1}^n\mathcal{K}_{si}w_{s}w_{i}}{\psi}+\frac{g'\varphi}{\cos\theta}+\alpha_{1}\\
   =&\,-\frac{2\sum_{s,q=1}^{n-1}w_{s}w_{q}\left(\frac{\beta_{q}}{\beta_{\nu}}\right)_{s}}{\psi}
      -\frac{2\sum_{s=1}^{n-1}\sum_{i=1}^n\mathcal{K}_{si}w_{s}w_{i}}{\psi}\\
  &\,+\frac{2(\sum_{s=1}^{n-1}w_{s}\beta_{s})(\sum_{s=1}^{n-1}w_{s}\beta_{s}\kappa_s)}{\beta_{n}^2\psi}+\frac{g'\varphi}{\cos\theta}+\alpha_{1}\\
 \geq&\;-C+\alpha_{1},
\end{aligned}
\end{equation}
where
\begin{align*}
\left|\frac{2(\sum_{s=1}^{n-1}w_{s}\beta_{s})(\sum_{s=1}^{n-1}w_{s}\beta_{s}\kappa_s)}{\beta_{n}^2\psi}\right|
\leq&\,\frac{2|(\sum_{s=1}^{n-1}w_{s}\beta_{s})(\sum_{s=1}^{n-1}w_{s}\beta_{s}\kappa_s)|}{\beta_0^4|Dw|^{2}}\\
\leq&\,\frac{C|\sum_{s=1}^{n-1}w_{s}\beta_{s}|^2}{\beta_0^4|Dw|^{2}}\\
\leq&\,\frac{C(\sum_{s=1}^{n-1}w_{s}^2)(\sum_{s=1}^{n-1}\beta_{s}^2)}{\beta_0^4|Dw|^{2}}\\
\le&\,C.
\end{align*}
Then we immediately arrive at a contradiction from taking $\alpha_{1}$ large enough.

$\mathbf{Case\;\mathrm{\mathbf{II}}}$: $x_{0}\in \Omega_{\mu}$. Recalling $w=u-\frac{\varphi d}{\cos\theta}$, direct calculation yields
\begin{equation}\label{eq07091}
\begin{aligned}
w_k=&\left(1-\frac{\varphi_zd}{\cos\theta}\right)u_k+O(1),\\
w_{ki}=&\left(1-\frac{\varphi_zd}{\cos\theta}\right)u_{ki}+O(d|Dw|^2+|Dw|+1),\\
w_{lij}=&\left(1-\frac{\varphi_zd}{\cos\theta}\right)u_{lij}+O(d|Dw|+1)|u_{ii}|+O(d|Dw|^3+|Dw|^2).
\end{aligned}
\end{equation}
Assume that $|Du|$ attains its boundary maximum at $y_{0}\in\partial\Omega$ and
\begin{equation}\label{314}
  |Du|^{2}(y_{0})\geq 4\sup_{\partial\Omega}\left|\frac{\varphi}{\cos\theta}\right|^{2}.
\end{equation}
As $\widetilde{G}(x_{0})\geq\widetilde{G}(y_{0})$
\begin{equation}\label{315}
\begin{aligned}
\psi(x_{0})\geq\frac{\beta_{0}^{2}C(\mu,\alpha_{1})e^{-\frac{1}{4}}}{4}|Du|^{2}(y_{0}),
\end{aligned}
\end{equation}
this follows that
\begin{equation}\label{316}
\psi(x_{0})
\geq\frac{\beta_{0}^{2}C(\mu,\alpha_{1})}{8C}|Du|^{2}(x_{0})
\geq\frac{\beta_{0}^{2}C(\mu,\alpha_{1})}{16C}|Dw|^{2}(x_{0})= C_{0}|Dw|^{2}(x_{0})>0,
\end{equation}
where $C_{0}\in(0,1)$. We choose coordinates such that the matrix $D^{2}u(x_{0})$ is diagonal. It is straightforward to observe that
$$\left|\overrightarrow{T}\right|\leq|Dw|, \quad \psi\leq\left|\overrightarrow{T}\right||Dw|,$$
where $T_{k}=\sum_{l=1}^{n}C^{kl}w_{l}$ and $\overrightarrow{T}=(T_{1},T_{2},\cdots,T_{n})$.
Then
\begin{equation}\label{318}
  C_{0}|Dw|\leq\left|\overrightarrow{T}\right|\leq|Dw|.
\end{equation}
Without loss of generality, we may assume the inequality $ T_{1}w_{1}\geq \frac{C_{0}}{n}|Dw|^{2}>0$ holds, which immediately yields $\frac{w_{1}}{T_1}\geq\frac{C_{0}}{n}$.
Now choose $\mu$ sufficiently small to guarantee that
\begin{equation}\label{321}
  \frac{u_{1}}{T_1}
  \geq\frac{C_{0}}{3n}.
\end{equation}
At $x_{0}$
\begin{equation}\label{323}
  0=\widetilde{G}_{i}=\frac{\psi_{i}}{\psi}+g'u_{i}+\alpha_{1}d_{i}.
\end{equation}
Denote ${C^{kl} }_{,i}=\frac{\partial C^{kl}}{\partial x_{i}}$ and ${C^{kl} }_{,ij}=\frac{\partial ^{2}C^{kl}}{\partial x_{i}\partial x_{j}}$. For $i=1$ in \eqref{323}, we obtain
\begin{equation}\label{326}
  \sum_{l=1}^{n}T_{l}w_{l1}=-\frac{\psi}{2}(g'u_{1}+\alpha_{1}d_{1})-\sum_{k,l=1}^{n}\frac{{C^{kl} }_{,1}}{2}w_{k}w_{l}.
\end{equation}
Combine \eqref{eq07091}, \eqref{318} and \eqref{326}
\begin{align*}
\left(1-\frac{\varphi_zd}{\cos\theta}\right)u_{11}=O(d|Du|^2+|Du|)\frac{\sum_k T_k}{T_1}-\frac{\psi}{2}\left(g'\frac{u_1}{T_1}+\frac{\alpha_1d_1}{T_1}\right)-\sum_{k,l}\frac{{C^{kl,} }_{1}}{2}\frac{w_{k}w_{l}}{T_1},
\end{align*}
which implies that $|u_{11}|\leq C|Du|^2$. And we can check that $u_{11}<0$ if $\mu$ is small enough.
By using the Cauchy-schwarz inequality at $x_0$
\begin{eqnarray}\label{329}
\begin{aligned}
0\geq F^{ij}\widetilde{G}_{ij}=&\;\frac{F^{ij}\left(\sum_{k,l=1}^{n}C^{kl}w_{k}w_{l}\right)_{ij}}{\psi}-2\alpha_{1}g'F^{ij}u_{i}d_{j}-\alpha_{1}^{2}F^{ij}d_{i}d_{j}\\
&\;+\left[g''-\left(g'\right)^{2}\right]F^{ij}u_{i}u_{j}+g'F^{ij}u_{ij}+\alpha_{1}F^{ij}d_{ij}\\
\geq&\;\frac{F^{ij}\left(\sum_{k,l=1}^{n}C^{kl}w_{k}w_{l}\right)_{ij}}{\psi}-2\alpha_{1}^{2}F^{ij}d_{i}d_{j}
+\alpha_{1}F^{ij}d_{ij}\\
&\;+\left[g''-2\left(g'\right)^{2}\right]F^{ij}u_{i}u_{j}+g'(k-l)f\\
=&\;\text{I}+\text{II}+\text{III}+\text{IV}+\text{V}.
\end{aligned}
\end{eqnarray}
According to (\ref{318}) and (\ref{321}), we obtain
\begin{equation}\label{330}
\begin{aligned}
\text{II}=&\;-2\alpha_{1}^{2}F^{ij}d_{i}d_{j}\geq-2\alpha_{1}^{2}\sum_{i=1}F^{ii},\\
\text{III}=&\;\alpha_{1}F^{ij}d_{ij}\geq-\alpha_{1}K_{0}\sum_{i=1}^{n}F^{ii},\\
\text{IV}=&\;\left[g''-2\left(g'\right)^{2}\right]F^{ij}u_{i}u_{j}\geq C_{2}\left[g''-2\left(g'\right)^{2}\right]|Dw|^{2}\sum_{i=1}^{n}F^{ii},\\
\text{V}=&\;g'(k-l)f\geq0,
\end{aligned}
\end{equation}
where $K_{0}$ is a positive constant related to the geometry of $\partial\Omega$. It remains to address the complex term \text{I}.
\begin{equation}\label{331}
\begin{aligned}
\text{I}=&\;\frac{F^{ij}\left(\sum_{k,l=1}^{n}C^{kl}w_{k}w_{l}\right)_{ij}}{\psi}\\
=&\;\frac{\sum_{i,j,k,l=1}^{n}F^{ij}{C^{kl}}_{,ij}w_{k}w_{l}}{\psi}
+\frac{2\sum_{i,j,k,l=1}^{n}F^{ij}C^{kl}w_{ijk}w_{l}}{\psi}\\
&\;\frac{4\sum_{i,j,k,l=1}^{n}F^{ij}{C^{kl}}_{,j}w_{ik}w_{l}}{\psi}+\frac{2\sum_{i,j,k,l=1}^{n}F^{ij}C^{kl}w_{ik}w_{jl}}{\psi}\\
=&\;I_{1}+I_{2}+I_{3}+I_{4}.
\end{aligned}
\end{equation}
For the term $I_{1}$, it is easy to know that
\begin{equation}\label{332}
  \psi I_{1}=\sum_{i,j,k,l=1}^{n}F^{ij}{C^{kl}}_{,ij}w_{k}w_{l}\geq-C|Dw|^{2}\sum_{i=1}^{n}F^{ii}.
\end{equation}
For the term $I_2$, by \eqref{eq0904203}, \eqref{eq07091}, \eqref{318}, Lemma \ref{le3} and the condition  $f^{\frac{1}{k-l}}\in C^{1,1}(\overline{\Omega}\times\mathbb{R})$ which is employed to address the lack of $\inf f$, we derive
\begin{equation}\label{333}
\begin{aligned}
\psi I_2
=&\;2\sum_{k=1}^{n}T_{k}\left[\partial_{k}f-\sum_{i,j=1}^{n}F^{ij}\left(\frac{\varphi d}{\cos\theta}\right)_{ijk}\right]\\
\geq&\;-Cd|Dw|^{4}\sum_{i=1}^{n}F^{ii}-C|Dw|^{3}\sum_{i=1}^{n}F^{ii}\\
&\;-Cd|Dw|^{2}\sum_{i=1}^{n}\left|F^{ii}u_{ii}\right|-C|Dw|\sum_{i=1}^{n}\left|F^{ii}u_{ii}\right|,
\end{aligned}
\end{equation}
similarly
\begin{equation}\label{334}
\psi I_3=4\sum_{i,j,k,l=1}^{n}F^{ij}{C^{kl}}_{,j}w_{ik}w_{l}
\geq-C|Dw|^{3}\sum_{i=1}^{n}F^{ii}-C|Dw|\sum_{i=1}^{n}\left|F^{ii}u_{ii}\right|,
\end{equation}
\begin{align}\label{335}
\psi I_4=2\sum_{i,j,k,l=1}^{n}F^{ij}C^{kl}w_{ik}w_{jl}
\geq&\;2\sum_{i=1}^{n}F^{ii}C^{ii}u_{ii}^{2}-Cd|Dw|^{2}\sum_{i=1}^{n}\left|F^{ii}u_{ii}\right|\\
\nonumber&\;-Cd|Dw|^{4}\sum_{i=1}^{n}F^{ii}-C|Dw|^{3}\sum_{i=1}^{n}F^{ii}.
\end{align}
Substituting \eqref{332}-(\ref{335}) into \eqref{331}, we obtain
\begin{equation}\label{336}
\begin{aligned}
\psi I
\geq&\;2\sum_{i=1}^{n}F^{ii}C^{ii}u_{ii}^{2}-Cd|Dw|^{2}\sum_{i=1}^{n}\left|F^{ii}u_{ii}\right|-C|Dw|\sum_{i=1}^{n}\left|F^{ii}u_{ii}\right|\\
&\;-Cd|Dw|^{4}\sum_{i=1}^{n}F^{ii}-C|Dw|^{3}\sum_{i=1}^{n}F^{ii}.
\end{aligned}
\end{equation}
Let
\begin{equation}\label{337}
  \mathcal{H}=2\sum_{i=1}^{n}F^{ii}C^{ii}u_{ii}^{2}-Cd|Dw|^{2}\sum_{i=1}^{n}\left|F^{ii}u_{ii}\right|-C|Dw|\sum_{i=1}^{n}|F^{ii}u_{ii}|,
\end{equation}
without loss of generality, we assume that  $C^{11}$ be the smallest of $\left\{C^{ii}\right\}_{i=1}^{n}$. This assumption implies that $C^{ii} \geq \frac{1}{2}$ for all $i \geq 2$, otherwise it follows that $\sum_{i=1}^{n}C^{ii}<n-1$, which contradicts with $\sum_{i=1}^{n}C^{ii}=n-1$. A direct computation shows that
\begin{equation}\label{339}
\mathcal{H}
\geq-C\sum_{i=1}^{n}F^{ii}\left(d|Dw|^{2}+|Dw|\right)^{2},
\end{equation}
thus
\begin{equation}\label{340}
  I
  \geq-Cd|Dw|^{2}\sum_{i=1}^{n}F^{ii}-C|Dw|\sum_{i=1}^{n}F^{ii}.
\end{equation}
Inserting (\ref{330}) and (\ref{340}) into (\ref{329}) yields
\begin{align}\label{341}
0\geq&\left\{C_{2}\left[g''-2\left(g'\right)^{2}\right]|Dw|^{2}-C_{3}\mu|Dw|^{2}-C|Dw|\right\}\sum_{i=1}^{n}F^{ii}\\
\nonumber\geq& \left(\frac{C_2}{128M^2}|Dw|^{2}-C_{3}\mu|Dw|^{2}-C|Dw|\right)\sum_{i=1}^{n}F^{ii}\\
\nonumber\geq&\left(\frac{C_2}{256M^2}|Dw|^{2}-C|Dw|\right)\sum_{i=1}^{n}F^{ii},
\end{align}
where the last inequality follows by taking $\mu$ sufficiently small. This completes the proof.
\end{proof}

\section{Second order estimates}
\subsection{Global second order estimates}
In this subsection, we establish global second order estimates for admissible solutions of problem \eqref{eq07081}. For degenerate Hessian quotient type equations, the regularity condition $f^{\frac{1}{k-l}}\in C^{1,1}(\overline{\Omega}\times\mathbb{R})$ is imposed to ensure that the estimate remain independent of $\inf f$.
\begin{lemma}\label{lm0904401}
Let $u\in C^{4}(\Omega)$ be an admissible solution of problem \eqref{eq07081}  in a bounded $C^{4}$ domain $\Omega\subset\mathbb{R}^{n}$. Suppose $0\leq f^{\frac{1}{k-l}}\in C^{1,1}(\overline{\Omega}\times\mathbb{R})$, then there exists a positive constant $C$ depending on $\overline{\Omega}$, $n$, $k$, $l$, $\gamma$, $|u|_{C^{1}}$, $|f^{\frac{1}{k-l}}|_{C^{1,1}}$ s.t.
\begin{equation}\label{eq2002}
  \sup_{\Omega}|D^{2}u|\leq \sup_{\partial\Omega}|D^{2}u|+C.
\end{equation}
\end{lemma}
\begin{proof}
Consider an auxiliary function for a vector $\xi\in \mathbb{R}^{n}$
\begin{equation*}
  V(x,\xi)=u_{\xi\xi}+\frac{1}{2}u_{\xi}^{2}.
\end{equation*}
Assume that $V$ takes its maximum at an interior point $x_0\in\Omega$ and a unit vector $\xi_0$. Without loss of generality, we can choose the coordinate system $e_{1},e_{2},\cdots,e_{n}$ at $x_0$ such that $e_1(x_0)=\xi_0$ and $\{u_{ij}(x_{0})\}$ is diagonal with $\max u_{ii}(x_0)=u_{11}(x_{0})$. Set
\begin{equation}\label{eq4004}
  \widetilde{V}=\;u_{11}+\frac{1}{2}u_{1}^{2},
\end{equation}
at $x_{0}$
\begin{align}
0=&\,\widetilde{V}_{i}=u_{11i}+u_{1}u_{1i}, \label{eq4005} \\
0\geq&\,\widetilde{V}_{ij}=u_{11ij}+u_{1i}u_{1j}+u_{1}u_{1ij}.\label{eq4006}
\end{align}
Recall the definition $\widetilde{f}=f^{\frac{1}{k-l}}$ and use \eqref{eq091902}, Lemma \ref{le3}
\begin{align}\label{eq0904401}
\partial_{11}f=&\,(k-l)f^{1-\frac{1}{k-l}}\partial_{11}\widetilde{f}+\left(1-\frac{1}{k-l}\right)\frac{(\partial_{1}f)^{2}}{f}\\
\nonumber\geq&\,-Cf^{1-\frac{1}{k-l}}-Cu_{11}f^{1-\frac{1}{k-l}}+\left(1-\frac{1}{k-l}\right)\frac{(\partial_{1}f)^{2}}{f}\\
\nonumber\geq&\,-C\sum_{i=1}^{n}F^{ii}-Cu_{11}\sum_{i=1}^{n}F^{ii}+\left(1-\frac{1}{k-l}\right)\frac{(\partial_{1}f)^{2}}{f}.
\end{align}
By \eqref{eq0904204}, \eqref{eq07111}, \eqref{eq4006}, \eqref{eq0904401} and the condition $f^{\frac{1}{k-l}}\in C^{1,1}$
\begin{equation}\label{eq4007}
\begin{split}
  0\geq F^{ij}\widetilde{V}_{ij}%
  =&\;F^{ij}u_{11ij}+F^{ij}u_{1i}u_{1j}+F^{ij}u_{1}u_{1ij}\\
  \geq&\,\partial_{11}f-F^{ij,rs}u_{ij1}u_{rs1}+Cu_{11}^2\sum_{i=1}^{n}F^{ii}+u_{1}\partial_1 f\\
  \geq&\;-C\sum_{i=1}^{n}F^{ii}-Cu_{11}\sum_{i=1}^{n}F^{ii}+u_{11}^{2}\sum_{i=1}^{n}F^{ii},
\end{split}
\end{equation}
which means that
\begin{equation}\label{eq4010}
  \sup_{\Omega}|D^{2}u|\leq\sup_{\partial\Omega}|D^{2}u|+C.
\end{equation}
\end{proof}
\subsection{Boundary second order estimates}
In this subsection, we establish boundary second order estimates, and divide the proof into three cases: double tangential derivatives, mixed tangential-oblique derivatives and double oblique derivatives.
The boundary function $G(x,u,Du)$ and the unit inner normal vector field $\nu$ have been extended to $\overline{\Omega}$ such that they remain constant along normals to $\partial\Omega$ in some neighbourhood $\Omega_{0}$ of the boundary.
\begin{lemma}\label{lemqx}
Let $u\in C^{2}(\overline{\Omega})$ be an admissible solution of problem \eqref{eq07081} in a bounded $C^{2}$ domain $\Omega\subset\mathbb{R}^{n}$, the oblique boundary condition $G\in C^{2}(\partial\Omega\times\mathbb{R}\times\mathbb{R}^{n})$. Then for any tangential vector field $\tau$
\begin{equation}\label{lm401}
  \sup_{\partial\Omega}|u_{\tau\beta}|\leq C,
\end{equation}
where $\beta=G_{p}(x,u,Du)$ and the positive constant $C$ depends on $\overline{\Omega},|G|_{C^{1}},|u|_{C^{1}}$.
\end{lemma}
\begin{proof}
Differentiating $G(x,u,Du)=0$ in problem \eqref{eq07081} with respect to $\tau$, we have
\begin{equation}\label{eq42011}
  D_{\tau}G=\left\langle DG,\tau\right\rangle=G_{x_{\tau}}+G_{z}u_{\tau}+G_{p_{k}}u_{k\tau}=0.
\end{equation}
Hence, we get $|u_{\tau\beta}|\leq C$ on $\partial\Omega$.
\end{proof}
For convenience, denote
\begin{equation*}
  M_{2}=\sup_{\Omega}\left|D^{2}u\right|,\quad M_{2}'=\sup_{\partial\Omega}\sup_{|\tau|=1,\left\langle\tau,\nu\right\rangle=0}\left|u_{\tau\tau}\right|.
\end{equation*}
Next, we estimate double tangential derivatives on the boundary.
\begin{lemma}\label{lemqq}
Let $u\in C^{4}(\overline{\Omega})$ be an admissible solution of problem \eqref{eq07081} in a bounded $C^{4}$ domain $\Omega\subset\mathbb{R}^{n}$. Suppose that the oblique boundary condition $G\in C^{2}(\partial\Omega\times\mathbb{R}\times\mathbb{R}^{n})$ satisfies \eqref{eq1002} with $G_{pp}\leq 0$ and $0\leq f^{\frac{1}{k-l}}\in C^{1,1}(\overline{\Omega}\times\mathbb{R})$. Then for any tangential vector field $\tau$ with $|\tau|\leq1$ and constant $\varepsilon>0$, we have
\begin{equation}\label{eq09054301}
  \sup_{\partial\Omega}u_{\tau\tau}\leq\varepsilon M_{2}+C_{\varepsilon},
\end{equation}
where $C_{\varepsilon}$ is a positive constant depending on $\varepsilon$, $\overline{\Omega}$, $n$, $k$, $l$, $\gamma$, $\beta_{0}$, $|u|_{C^{1}}$, $|f^{\frac{1}{k-l}}|_{C^{1,1}}$ and $|G|_{C^{2}}$.
\end{lemma}
\begin{proof}
Suppose a function
\begin{equation*}\label{eq402}
  v_{\tau}=u_{\tau\tau}+\frac{1}{2}|u_{\tau}|^{2}
\end{equation*}
attains a maximum at $y_{0}\in\partial\Omega, \tau=\tau_{0}$.  Without loss of generality, we assume $y_{0}=0$ and $\tau_{0}=e_{1}=(1,0,\cdots,0)$.
Set $b=\frac{\nu_{1}}{\left\langle\beta,\nu\right\rangle}$ and $\tau=e_{1}-b\beta$, 
thus at any point on $\partial\Omega$
\begin{equation*}\label{eq403}
v_{1}
=v_{\tau}+b\left(2u_{\tau\beta}+u_{\beta}u_{\tau}\right)+b^{2}\left(u_{\beta\beta}+\frac{1}{2}u_{\beta}^{2}\right),
\end{equation*}
and $v_{1}(0)=v_{\tau}(0)$, $b(0)=0$, $\tau(0)=e_{1}$.
Define $h=\frac{1}{\left\langle\beta,\nu\right\rangle}(2u_{\tau\beta}+u_{\beta}u_{\tau})$. From \eqref{eq42011} on $\partial\Omega$, we obtain
\begin{align*}
  \left|h(x)-h(0)\right|=&\, |Dh(tx)||x-0|\\
  =&\,\left|D\left(\frac{1}{\langle\beta,\nu\rangle}\right)(2u_{\tau \beta}+u_\beta u_\tau)+\frac{1}{\langle\beta,\nu\rangle}D(2u_{\tau \beta}+u_\beta u_\tau)\right||x|\\
  \leq&\,C(|u_{\tau \beta}|+|u_\beta u_\tau|)|x|+C|D(u_{\tau \beta}+u_\beta u_\tau)||x|\\
  \leq&\, C(1+M_2)|x|+C|D(G_{x_\tau}+G_{z}u_{\tau})||x|\\
  \leq&\, C(1+M_{2})|x|,
\end{align*}
where $t\in (0,1)$. Then
\begin{align}\label{eq405}
 \nonumber v_{1}-h(0)\nu_{1}=&\;v_{\tau}+h\nu_{1}+b^{2}\left(u_{\beta\beta}+\frac{1}{2}u_{\beta}^{2}\right)-h(0)\nu_{1}\\
              \leq&\;v_{\tau}+|h-h(0)||\nu_{1}|+b^{2}\left|u_{\beta\beta}+\frac{1}{2}u_{\beta}^{2}\right|\\
 \nonumber \leq&\; v_{\tau}+C_{1}(1+M_{2})|x|^{2},
\end{align}
where $C_{1}$ is a positive constant depending on $\beta_{0}$, $\overline{\Omega}$, $|u|_{C^1}$, $|G|_{C^{2}}$.

Next, we consider
\begin{equation*}\label{eq406}
  \widetilde{v}_{1}=v_{1}-h(0)\nu_{1}-C_{1}(1+M_{2})|x|^{2},
\end{equation*}
where the constant $C_1$ is the same quantities in \eqref{eq405} satisfying
\begin{equation}\label{eq407}
\widetilde{v}_{1}\leq\; v_{\tau}=u_{\tau\tau}+\frac{1}{2}|u_{\tau}|^{2}
\leq|\tau|^{2}v_{1}(0)
\leq Qv_{1}(0),
\end{equation}
$Q$ is a non-negative function in $C^{2}(\overline{\Omega})$ satisfying $Q\geq|\tau|^{2}$ on $\partial\Omega$ and $Q(0)=1$.
Note that
\begin{equation*}\label{eq408}
|\tau|^{2}=|e_{1}-b\beta|^2 
\leq\;1-2\frac{\beta_1}{\cos\theta}(0)\nu_1 +C_{1}|x|^2=Q.
\end{equation*}

Finally, we choose a new auxiliary function
\begin{equation*}\label{eq412}
  \overline{v}=\widetilde{v}_{1}-v_{1}(0)Q-K(1+M_{2})\phi,
\end{equation*}
where $\phi\in C^{2}(\overline{\Omega})$ is a function in $\Omega$ satisfying $\phi=0$, $\phi_{\nu}=-1$ and $\phi_{\tau}=0$ on $\partial\Omega$, $K$ is a sufficiently large constant.
Now differentiating $G(x,u,Du)=0$ twice in direction $\tau$ with $\tau(0)=e_1$ and using $G_{pp}\leq0$, we obtain
\begin{equation*}\label{eq410}
D_{\beta}u_{11}(0)\geq\;-C\left(1+M_{2}\right).
\end{equation*}
At $y_0=0$
\begin{equation*}\label{eq413}
  D_{\beta}\overline{v}=D_{\beta}\left[\widetilde{v}_{1}-v_{1}(0)Q\right]-K(1+M_{2})\phi_{\beta}>0
\end{equation*}
with $K>\frac{C_{1}}{\beta_{0}}$. The specific details can be referred to Jiang-Trudinger \cite{JT1,JT3}. On $\partial\Omega$
$\phi=0$, from \eqref{eq407}
$$\overline{v}(0)=0, \quad \overline{v}\leq0,\, \text{on}\,\, \partial\Omega,$$
hence $\overline{v}$ takes its maximum at $\overline{y}_{0}\in\Omega$. Suppose that the matrix $\{u_{ij}(\overline{y}_{0})\}$ is diagonal, at $\overline{y}_{0}$
\begin{align}\label{eq414}
 \nonumber0\geq L\overline{v}
 =&\;F^{ij}u_{11ij}+F^{ij}u_{1i}u_{1j}+u_1F^{ij}u_{1ij}-h(0)F^{ij}\nu_{1ij}-2C_{1}(1+M_{2})\sum_{i=1}^{n}F^{ii}\\
 \nonumber &\; +2\frac{\beta_1}{\left\langle\beta,\nu\right\rangle}(0)v_1(0)F^{ij}\nu_{1ij}-2C_1v_1(0)\sum_{i=1}^{n}F^{ii}-K(1+M_2)F^{ij}\phi_{ij}\\
   \geq&\;-C(1+M_{2})\sum_{i=1}^{n}F^{ii}+Cu_{11}^{2}\sum_{i=1}^{n}F^{ii},
\end{align}
this implies $ u_{11}(\overline{y}_{0})\leq\sqrt{C(1+M_{2})}\leq\varepsilon M_{2}+C_{\varepsilon}$,
where $\varepsilon$ is a positive constant and $C_{\varepsilon}$ depends on $\varepsilon$, $n$, $k$, $l$, $\gamma$, $\beta_{0}$, $\overline{\Omega}$, $|u|_{C^{1}}$, $|f^{\frac{1}{k-l}}|_{C^{1,1}}$, $|G|_{C^2}$.
Choosing $\phi\geq-\frac{\varepsilon}{4K}$ and $Q>\frac{1}{2}$, we get
\begin{equation}\label{eq416}
  v_{1}(0)\leq\varepsilon M_{2}+C_{\varepsilon}.
\end{equation}
Since $v_{\tau}$ attains the maximum at $0$ and $\tau=e_{1}$ on $\partial\Omega$, then $v_{\tau}\leq v_{1}(0)\leq\varepsilon M_{2}+C_{\varepsilon}$,
which yields that
\begin{equation}\label{eq418}
  \sup_{\partial\Omega}u_{\tau\tau}\leq\varepsilon M_{2}+C_{\varepsilon}.
\end{equation}
\end{proof}

Now it remains to establish the double oblique derivatives on the boundary.
\begin{lemma}\label{lemxx}
Let $u\in C^{3}(\overline{\Omega})$ be an admissible solution of problem \eqref{eq07081} in a bounded $C^{3}$ domain $\Omega\subset\mathbb{R}^{n}$. Suppose that the oblique boundary condition $G\in C^{2}(\partial\Omega\times\mathbb{R}\times\mathbb{R}^{n})$ satisfies \eqref{eq1002} with $G_{pp}=0$ and $0\leq f^{\frac{1}{k-l}}\in C^{1,1}(\overline{\Omega}\times\mathbb{R})$. Then for a small constant $\varepsilon>0$
\begin{equation}\label{lm402}
  \sup_{\partial\Omega}u_{\beta\beta}\leq\varepsilon M_{2}+C_{\varepsilon}\left(1+M_{2}'\right),
\end{equation}
where $C_{\varepsilon}$ is a positive constant depending on $\varepsilon$, $n$, $k$, $l$, $\gamma$, $\beta_{0}$, $\overline{\Omega}$, $|u|_{C^{1}}$, $|f^{\frac{1}{k-l}}|_{C^{1,1}}$ and $|G|_{C^{2}}$.
\end{lemma}
\begin{proof}
For any fixed boundary point $x_{0}\in\partial\Omega$, we consider the function
\begin{equation}\label{eq4201}
  \overline{V}=G(x,u,Du)+\frac{a_{1}}{2}|Du-Du(x_{0})|^{2},
\end{equation}
here $a_{1}\leq1$ is a positive constant.
 Since $\widetilde{F}$ is homogeneous of degree $1$ and $G_{pp}=0$
\begin{equation}\label{eq4204}
\begin{aligned}
  \widetilde{L}\overline{V}
  =&\;\widetilde{F}^{ij}G_{x_{i}x_{j}}+2\widetilde{F}^{ij}G_{x_{i}z}u_{j}+2\widetilde{F}^{ij}G_{x_{i}p_{k}}u_{jk}+G_{zz}\widetilde{F}^{ij}u_{i}u_{j}\\
  &\;+2\widetilde{F}^{ij}G_{zp_{k}}u_{i}u_{jk}+G_{z}\widetilde{F}^{ij}u_{ij}+G_{p_{k}p_{l}}\widetilde{F}^{ij}u_{ik}u_{lk}+G_{p_{k}}\widetilde{F}^{ij}u_{ijk}\\
  &\;+a_{1}\widetilde{F}^{ij}u_{ik}u_{jk}+a_{1}(u_{k}-u_{k}(x_{0}))\widetilde{F}^{ij}u_{ijk}\\
 \geq&\;-C\sum_{i=1}^{n}\widetilde{F}^{ii}+a_{1}\widetilde{F}^{ij}u_{ik}u_{jk}+2\widetilde{F}^{ij}G_{x_{i}p_{k}}u_{jk}+2\widetilde{F}^{ij}G_{zp_{k}}u_{i}u_{jk}\\
 \geq&\;-C\left(1+\frac{1}{a_{1}}\right)\sum_{i=1}^{n}\widetilde{F}^{ii},  
\end{aligned}
\end{equation}
where $\widetilde{L}=\widetilde{F}^{ij}\partial_{ij}$ and we use the Cauchy-Schwarz inequality in the last inequality.

Now, a suitable upper barrier for $\overline{V}$ at the point $x_{0}\in\partial\Omega$ can be constructed as
\begin{equation*}\label{eq4205}
  \overline{\phi}=-c\phi+\frac{a_{2}}{2}|x-x_{0}|^{2},\quad\quad\text{in}\;\Omega_{\mu},
\end{equation*}
with $\phi=-d+K_{3}d^{2}$,
where $\Omega_{\mu}:=\{x\in\overline{\Omega}|\;0<d(x)\leq\mu\}$ and $K_{3}$, $c$, $a_{2}$, $\mu$ are positive constants. Choose $\mu$, $\delta>0$ small enough and $K_3$ large enough, s.t.
\begin{equation*}\label{eq4207}
            \widetilde{L}  \overline{\phi}
              = -c\widetilde{F}^{ij}\phi_{ij}+a_{2}\widetilde{F}^{ij}\delta_{ij}
              \leq(a_{2}-c\delta)\sum_{i=1}^{n}\widetilde{F}^{ii},\quad\text{in}\;\Omega_{\mu},
\end{equation*}
then we obtain $ \widetilde{ L} \left(\overline{V}-\overline{\phi}\right)\geq 0$ in $v$ by choosing $c\geq\frac{a_{2}}{\delta}+\frac{C}{\delta}\left(1+\frac{1}{a_{1}}\right)$.

Next, we deal with $\overline{V}$ and $\overline{\phi}$ on the boundary of $\Omega_{\mu}$. For any $x\in\partial\Omega$
$$\overline{V}=\frac{a_{1}}{2}|Du(x)-Du(x_{0})|^{2}, \quad \overline{\phi}=\frac{a_{2}}{2}|x-x_{0}|^{2},$$
thus
\begin{equation}\label{eq4209}
\begin{split}
  |Du(x)-Du(x_{0})|\leq C(1+M_{2}')|x-x_{0}|,
\end{split}
\end{equation}
consequently
\begin{equation}\label{eq4210}
\begin{split}
  |Du(y)-Du(x_{0})|^{2}
  \leq C[1+(M_{2}')^{2}+M_{2}]\left(|y'-x_{0}|^{2}+d\right),
\end{split}
\end{equation}
where $y\in\Omega_{\mu}$ and $y'$ is the closest point on $\partial\Omega$.
Choose $a_{2}=a_{1}[1+(M_{2}')^{2}+M_{2}]$ and $c\geq a_{2}$ s.t.
\begin{equation}\label{eq09064301}
  \overline{V}-\overline{\phi}\leq0,\quad\text{on}\;\overline{\Omega}_{\mu}\cap\left\{G(x,z,p)=0\right\}\supset\partial\Omega.
\end{equation}
On $\partial\Omega_{\mu}\cap\Omega$, we have $\overline{V}-\overline{\phi}\leq0$ if $c\geq\frac{C}{\mu}$.
Therefore, $\overline{V}-\overline{\phi}\leq0$ in $\Omega_{\mu}$.
Since $\overline{V}(x_{0})=\overline{\phi}(x_{0})=0$
$$D_{\beta}\overline{V}(x_{0})\leq D_{\beta}\overline{\phi}(x_{0}),$$
which implies
\begin{equation*}
  u_{\beta\beta}(x_{0})\leq c\sup_{\partial\Omega}\left\langle\nu,\beta\right\rangle+C.
\end{equation*}
Finally, we fix the constant $c$
\begin{equation}\label{eq4215}
\begin{aligned}
  c\leq\frac{a_{2}+C(1+\frac{1}{a_{1}})}{\delta}+a_{2}+\frac{C}{\mu}
    \leq\frac{C}{\delta}\varepsilon_{1}M_{2}+C_{\varepsilon_{1}}(1+M_{2}'),
\end{aligned}
\end{equation}
where $C$ depends on $n$, $k$, $l$, $\gamma$, $\beta_{0}$, $\overline{\Omega}$, $|u|_{C^{1}}$, $|f^{\frac{1}{k-l}}|_{C^{1,1}}$ and $|G|_{C^{2}}$. For any $\varepsilon>0$, taking $a_{1}=\frac{1}{1+\varepsilon_{1}M_{2}}$ and $\varepsilon_{1}\leq\frac{\varepsilon\delta}{C\sup_{\partial\Omega}\left\langle\nu,\beta\right\rangle}$, we get
\begin{equation}\label{eq4216}
  u_{\beta\beta}(x_{0})\leq\varepsilon M_{2}+C_{\varepsilon}(1+M_{2}').
\end{equation}
\end{proof}

Based on the preceding lemmas, we directly arrive at the following global $C^{2}$ estimates.
\begin{lemma}\label{C2}
Let $u\in C^{4}(\overline{\Omega})$ be an admissible solution of the boundary value problem \eqref{eq07081} in a bounded $C^{4}$ domain $\Omega\subset\mathbb{R}^{n}$. Suppose that the oblique boundary condition $G\in C^{2}(\partial\Omega\times\mathbb{R}\times\mathbb{R}^{n})$ satisfies \eqref{eq1002} with $G_{pp}=0$ and $0\leq f^{\frac{1}{k-l}}\in C^{1,1}(\overline{\Omega}\times\mathbb{R})$. Then
\begin{equation}\label{eq4301}
  \sup_{\overline{\Omega}}|D^{2}u|\leq C,
\end{equation}
where $C$ is a positive constant depending on  $\overline{\Omega}$, $n$, $k$, $l$, $\gamma$, $\beta_{0}$, $|u|_{C^{1}}$, $|f^{\frac{1}{k-l}}|_{C^{1,1}}$, $|G|_{C^{2}}$.
\end{lemma}
\begin{proof}
From Lemma \ref{lm0904401}-Lemma \ref{lemxx}, we conclude that
\begin{align*}
\sup_{\partial\Omega}u_{\beta\beta}\leq \varepsilon_3 M_2+C_{\varepsilon_3}(1+\varepsilon_1 M_2+C_{\varepsilon_1})\leq \varepsilon M_2+C_\varepsilon,
\end{align*}
and $\sup_{\partial\Omega}|D^2 u|\leq \varepsilon M_2+C_\varepsilon$. Then
\begin{align*}
M_2\leq\sup_{\partial\Omega}|D^2 u|+C\leq \varepsilon M_2+C_\varepsilon,
\end{align*}
which implies $\sup_{\overline{\Omega}}|D^{2}u|\leq C$.
\end{proof}
Finally, we complete the proof of Theorem \ref{Th1}. From the maximal principle, $C^0$ estimates can be directly obtained by the existence of the admissible subsolution and $f_u\geq 0,\ \varphi_u>0$.
Then based on the a priori estimates for non-degenerate case, $C^{2,\alpha}\;(0<\alpha<1)$ estimates can be established by Evans-Krylov theory and higher order estimates followed by Schauder theory. Then the existence result can be derived by the continuity method and the uniqueness assertion is immediate from the maximum principle, more details see \cite{GT1997}. Theorem \ref{Th1} can be proved by approximation as in \cite{JW2022}.


\end{document}